\documentclass[a4paper,12pt,reqno]{amsart}
\usepackage{amsfonts}
\usepackage{amsmath}
\usepackage{amssymb}
\usepackage[a4paper]{geometry}
\usepackage{mathrsfs}
\usepackage[colorlinks]{hyperref}
\renewcommand\eqref[1]{(\ref{#1})} 
\numberwithin{equation}{section}
\theoremstyle{plain}
\newtheorem{thm}{Theorem}[section]
\newtheorem{prop}[thm]{Proposition}
\newtheorem{cor}[thm]{Corollary}

\theoremstyle{definition}
\newtheorem{defn}[thm]{Definition}
\newtheorem{rem}[thm]{Remark}

\begin{document}

   \title[On Green functions for Dirichlet sub-Laplacians on
    $H$-type groups]
   {On Green functions for Dirichlet sub-Laplacians on H-type
   groups}

\author[Nicola Garofalo]{Nicola Garofalo}
\address{
  Nicola Garofalo:
  \endgraf
  Department of Civil and Environmental Engineering (DICEA)
 \endgraf
  University of Padova
  \endgraf
  Via Marzolo, 9 - 35131 Padova,
  \endgraf
  Italy
  \endgraf
  {\it E-mail address} {\rm nicola.garofalo@unipd.it}
  }
\author[Michael Ruzhansky]{Michael Ruzhansky}
\address{
  Michael Ruzhansky:
  \endgraf
  Department of Mathematics
  \endgraf
  Imperial College London
  \endgraf
  180 Queen's Gate, London SW7 2AZ
  \endgraf
  United Kingdom
  \endgraf
  {\it E-mail address} {\rm m.ruzhansky@imperial.ac.uk}
  }
\author[Durvudkhan Suragan]{Durvudkhan Suragan}
\address{
  Durvudkhan Suragan:
  \endgraf
  Institute of Mathematics and Mathematical Modelling
  \endgraf
  125 Pushkin str.
  \endgraf
  050010 Almaty
  \endgraf
  Kazakhstan
  \endgraf
  and
  \endgraf
  Department of Mathematics
  \endgraf
  Imperial College London
  \endgraf
  180 Queen's Gate, London SW7 2AZ
  \endgraf
  United Kingdom
  \endgraf
  {\it E-mail address} {\rm d.suragan@imperial.ac.uk}
  }

\thanks{The authors were supported in parts by the EPSRC
 grant EP/K039407/1 and by the Leverhulme Grant RPG-2014-02,
 as well as by the MESRK grant 5127/GF4. No new data was collected or generated during the course of this research.}


     \begin{abstract}
     We construct Green functions of Dirichlet boundary value problems for sub-Laplacians on certain unbounded domains of a prototype Heisenberg-type group (prototype $H$-type group, in short). We also present solutions in an explicit form of the Dirichlet problem for the sub-Laplacian with non-zero boundary datum on half-space, quadrant-space, etc, as well as in a strip.
     \end{abstract}
     \maketitle

\section{Introduction}

Prototype $H$-type groups are an important class of
homogeneous stratified Lie groups of step two since
any (abstract) $H$-type group is naturally isomorphic
to a prototype $H$-type group.
The abelian group $(\mathbb{R}^{d}; +)$ and
the Heisenberg group $\mathbb{H}^{d}$ are
examples of the prototype $H$-type groups.
We denote a prototype $H$-type group by $\mathbb{G}$ and by
$\mathcal{L}$ a sub-Laplacian on $\mathbb{G}$.
We consider a smooth open set $D\subset \mathbb{G}$
with boundary $\partial D$, and study the
Dirichlet problem for its sub-Laplacian

\begin{equation}\label{LD}
\bigg\{\begin{matrix}
\mathcal{L}u=f\quad {\rm in}\; D, \\
u=\phi\quad {\rm on}\; \partial D.
\end{matrix}
\end{equation}

In Euclidean (elliptic) case the boundary value problem \eqref{LD} for suitable essential class of functions $f, \phi$ (say, $f\in C^{\alpha}(D),\,\alpha>0,$ and $\phi\in C(\partial D)$) has a classical solution, that is, a solution in $C^{2}(D)\cap C^{1}(\overline{D})$. In general, this fact fails completely for the hypoelliptic boundary value problem \eqref{LD}. The example of D. Jerison \cite{J} (see also \cite[Section 4]{CGH08}) shows that even if the domain $D$ and the boundary datum $\phi$ (with $f\equiv0$) are real analytic in the Heisenberg group $\mathbb{H}^{d}$, then the solution of the Dirichlet problem \eqref{LD} may be not better than H\"older continuous near a characteristic boundary point, that is, the solution is not classical.
We recall that the characteristic set (related to vector fields
$\{X_{1},...,X_{m}\}$) of $D$ is the set
$$\{x\in\partial D|\,X_{k}(x)\in T_{x}(\partial D),\,k=1,...,m\},$$
$T_{x}(\partial D)$ being the tangent space to $\partial  D$ at the point $x$.
Here vector fields $X_{1},...,X_{m}$ with their commutators span the Lie algebra of $\mathbb{G}$. See Section \ref{SEC:2} for more details.

\medskip

 The main aim of this short note is to give an answer to
 the question:
 {\em Is there a class of domains in which the Dirichlet
 boundary value problem \eqref{LD} is explicitly solvable
 in the classical sense?}

\medskip

This question is inspired by M. Kac's question: {\em Is there any boundary value problem for the Laplacian which is explicitly solvable
in the classical sense for any smooth domain?}

An answer to M. Kac's question was given
in our recent paper \cite{Ruzhansky-Suragan:Kohn-Laplacian}
for the Heisenberg group and in 
\cite{Ruzhansky-Suragan:Layers} for general homogeneous stratified Lie groups.
The boundary conditions appearing there are, however,
nonlocal and the corresponding boundary value problem can be called Kac's boundary value problem.
It is interesting to note that the explicit solutions in these papers
have been constructed for Kac's boundary value problem for the sub-Laplacian equally
well also in the presence of characteristic points on the boundary.

However, the above D. Jerison's example hints that one should seek for an answer to
our question concerning the Dirichlet problem \eqref{LD}
among domains without characteristic points. Even ball-like bounded domains of non-Abelian $H$-type groups have non-empty collection (set) of characteristic points. For example, any bounded domain of class $C^{1}$ in the Heisenberg group $\mathbb{H}^{d}$, whose boundary is homeomorphic to the $2d$-dimensional sphere $\mathbb{S}^{2d}$, has non-empty characteristic set (see, for example, \cite{DGN}). In general, this implies that domains, which we are looking for, should be unbounded. On the other hand, to give an explicit representation of a solution we also need to construct Green functions for these domains, so the domains need to have sufficiently rich symmetry. Thus, in this note we show that the boundary value problem \eqref{LD} is explicitly solvable in the classical sense in such domains as half-spaces, quadrant-spaces and so on.
Our analysis is based on the Euclidean ideas combined with further results on $H$-type groups (and on more general groups) obtained by Kohn-Nirenberg \cite{KN65}, Folland \cite{Fol75} and Kaplan \cite{Kap80}. We also should note that this short paper is partially motivated by the recent work \cite{DKM16} in which the authors construct a Green function for the Neumann sub-Laplacian on the Koranyi ball of the Heisenberg group.

We refer to \cite{GL03}, \cite{GN88}, \cite{GW92}, \cite{LU97},  \cite{Ruzhansky-Suragan:squares} and \cite{WN2016} as well as to references therein for more general Green function analysis of second order subelliptic (and weighted degenerate) operators.

We also refer to \cite{FR16} for a general point of view on $H$-type groups from the perspective of general stratified/graded/homogeneous/nilpotent Lie groups. For functional inequalities on stratified Lie groups and further literature review we can refer to \cite{RS17a}.

In Section \ref{SEC:2} we very briefly review the main concepts of (prototype) $H$-type groups and fix the notation. In Section \ref{Sec3} we construct Green functions and give representation formulae for solutions.

\section{Preliminaries}
\label{SEC:2}
Following Bonfiglioli, Lanconelli and Uguzzoni \cite{BLU07} we briefly recall the main notions concerning prototype $H$-type groups. We adopt the notation from \cite{BLU07} and refer to it for further details.
\begin{defn}\label{DEF:proH}
The space $\mathbb{R}^{m+n}$ equipped with the group law
 \begin{equation}
     (x,t)\circ(y,\tau)=\left(\begin{matrix}
         x_{k}+y_{k},\quad k=1,...,m \\
         t_{k}+\tau_{k}+\frac{1}{2}\langle A^{(k)}x,y \rangle,\quad k=1,...,n
        \end{matrix}\right)
  \end{equation}
and with the dilation $\delta_{\lambda}(x,t)=(\lambda x, \lambda^{2}t)$
is called a {\em prototype $H$-type group}.
Here $A^{(k)}$ is an $m\times m$ skew-symmetric orthogonal matrix, such that,
$A^{(k)}A^{(l)}+A^{(l)}A^{(k)}=0$ for all $k,l\in \{1,...,n\}$ with $k\neq l$.
\end{defn}
Throughout this paper we use the notation $\mathbb G$ for a prototype $H$-type group
$(\mathbb{R}^{m+n},\delta_{\lambda},\circ)$. Note that any (abstract) $H$-type group is naturally isomorphic to a prototype $H$-group (see \cite[Theorem 18.2.1]{BLU07}).  It can be directly checked that
the group operation $\circ$ defines a step two nilpotent Lie group in which the inverse of
$(x,t)$ is $(-x,-t)$, that is, the identity is the origin.
It  can be also verified that the vector
field in the Lie algebra $\mathfrak{g}$ of $\mathbb{G}$ that agrees at the origin with $\frac{\partial}{\partial x_{j}},\; j=1,...,m$, is given by
 \begin{equation}\label{vectorfields}
X_{j}=\frac{\partial}{\partial x_{j}}+\frac{1}{2}\sum_{k=1}^{n}\left( \sum_{i=1}^{m}a^{k}_{j,i}x_{i}\right)\frac{\partial}
{\partial t_{k}},
\end{equation}
where $a^{k}_{j,i}$ is the $(j,i)$ element of
the matrix $A^{(k)}$.
The vector fields $X_{1},...,X_{m}$
with their commutators span the whole $\mathfrak{g}$.
Thus, the sub-Laplacian on $\mathbb G$ is given by

 \begin{equation}\label{subLaplacian}
 \mathcal{L}=\sum_{j=1}^{m}X_{j}^{2}=\Delta_{x}+\frac{1}{4}|x|^{2}\Delta_{t}+\sum_{k=1}^{n}\langle A^{(k)}x,\nabla_{x}\rangle\frac{\partial}{\partial t_{k}},
 \end{equation}
where $\Delta$ and $\nabla$ are the Euclidean Laplacian and the Euclidean gradient, respectively.
It is not restrictive to suppose that, if $\varrho$ is the center of $\mathfrak{g}$, $\varrho^{\perp}$ is the orthogonal complement of $\varrho$ and
$$m=\dim(\varrho^{\perp}),\quad n=\dim(\varrho).$$
We have that the homogeneous dimension of the group is
$Q=m+2n.$ We note that since for $H$-type groups we have $m\geq 2$ and $n\geq 1$, we actually always have $Q\geq 4$.

Now using a generic coordinate $\xi\equiv(x,t),\,x\in\mathbb{R}^{m},\,t\in\mathbb{R}^{n}$, let us introduce the following functions on $\mathbb{G}$:
$$ v:\mathbb{G}\rightarrow \varrho^{\perp},\quad
v(\xi):=\sum_{j=1}^{m}\langle\exp^{-1}_{\mathbb{G}}(\xi),\,X_{j}\rangle X_{j},$$
where $\{X_{1},...,X_{m}\}$ is an orthogonal basis of $\varrho^{\perp}$,
$$ z:\mathbb{G}\rightarrow \varrho,\quad
 z(\xi):=\sum_{j=1}^{n}\langle\exp^{-1}_{\mathbb{G}}(\xi),\,Z_{j}\rangle Z_{j},$$
where $\{Z_{1},...,Z_{n}\}$ is an orthogonal basis of $\varrho.$
Thus, by the definition of $ v$ and $ z$, for any $\xi\in\mathbb{G}$, one has
$$
\xi=\exp(v(\xi)+z(\xi)),\quad v(\xi)\in\varrho^{\perp},\quad z(\xi)\in\varrho,
$$
and a direct calculation shows  (see, e.g. \cite[Proof of Remark 18.3.3]{BLU07}) that
$$|v(\xi)|=|x|,\quad | z(\xi)|=|t|.$$
It simplifies A. Kaplan's theorem in the following form
\begin{thm}\label{kaplanthm}
There exists a positive constant $c$ such that
\begin{equation}\label{fundsolution}
\Gamma(\xi):=c\left(|x|^{4}+16|t|^{2}\right)^{(2-Q)/4}
\end{equation}
is the fundamental solution of the sub-Laplacian, that is,
\begin{equation}\label{fundsoldelta}
\mathcal{L}\Gamma_{\zeta}=-\delta_{\zeta},
\end{equation}
where $\Gamma_{\zeta}(\xi)=\Gamma(\zeta^{-1}\circ\xi)$ and $\delta_{\zeta}$ is the Dirac distribution
at $\zeta\equiv (y,\tau)\in \mathbb{G}$.
\end{thm}
Note that more general result for abstract $H$-type groups was established by A. Kaplan in \cite{Kap80}.
Now the Green function for the Dirichlet sub-Laplacian in $D$ is defined by the formula
\begin{equation}\label{greenfunction}
G_{D}(\xi,\zeta)=\Gamma(\zeta^{-1}\circ\xi)-h_{\zeta}(\xi),
\end{equation}
with
\begin{equation}\label{boundary}
G_{D}(\xi,\zeta)=0,\quad \xi\in\partial D.
\end{equation}
Here, $h_{\zeta}(\xi)$ is a harmonic function, that is,
\begin{equation}\label{sub-harmonic}
\mathcal{L}h_{\zeta}(\xi)=0\quad {\rm in}\;D,
\end{equation}
having as boundary values (in the Perron-Wiener-Brelot sense) the fundamental solution with pole at $\zeta\in D$.

Let $\partial D$ be the boundary of a smooth domain $D$ in $\mathbb{G}$, $d\nu$ the volume element on $\mathbb{G}$, and $\langle X_{j}, d\nu\rangle$ the natural pairing between vector fields and differential forms. We also recall that the standard Lebesque measure on $\mathbb R^{m+n}$ is the Haar measure for $\mathbb{G}$ (see, e.g. \cite[Proposition 1.6.6]{FR16}).

The following version of Green's second formula will be useful for our analysis.
It goes back to the integration by parts formula in \cite{CGH08} but the following form
(given in \cite[Proposition 3.10]{Ruzhansky-Suragan:Layers}) will be useful for us here.
 
\begin{prop}[Green's second formula]
\label{green2}
Let $u,v\in C^{2}(D)\bigcap C^{1}(\overline{D}).$ Then
\begin{equation}\label{g2}
\int_{D}(u\mathcal{L}v-v\mathcal{L}u)d\nu
=\int_{\partial D}(u\langle\widetilde{\nabla}
v,d\nu\rangle-v\langle \widetilde{\nabla}  u,d\nu\rangle),
\end{equation}
where $\mathcal{L}$ is the sub-Laplacian on $\mathbb{G}$ and
\begin{equation}\label{nabla}
\mathcal{\widetilde{\nabla} }u=\sum_{k=1}^{m}
\left(X_{k}u\right)X_{k}.
\end{equation}
\end{prop}

By classical arguments one verifies that the above Green's second formula is valid for
the Green function (and the fundamental solution). See \cite{Ruzhansky-Suragan:Layers} for futher discussion, but also \cite[Section 7]{CGH08}. The relation between the $(n-1)$-form under the integral in the right-hand side of \eqref{g2} and the perimeter and surface measures
on $\partial D$ has been discussed in \cite{Ruzhansky-Suragan:Layers}.

\section{Green functions and representations of solutions}
\label{Sec3}

In this section we construct, by using the classical method of reflection, Green functions of Dirichlet boundary value problems for sub-Laplacians on $l$-wedge like and $l$-strip like unbounded domains of a prototype $H$-type group. We also present solutions in an explicit form of the Dirichlet problem for the sub-Laplacian with non-zero boundary datum on those domains. Of course, the results are well known in the abelian cases.

\subsection{Green functions and representations of solutions in $l$-wedge like spaces.}
Let $\mathbb{G}^{\ddagger}$ be the $l$-wedge like space 
$$\mathbb{G}^{\ddagger}=\{\xi=(x_{1},\ldots,x_{m},t_{1},\ldots,t_{n})|\;x_{1},\ldots,x_{l}>0\},$$ 
for some
$1\leq l\leq m$. Let the point $\zeta=(y,\tau)=(y_{1},y_{2},\ldots,y_{m},\tau_{1},\ldots,\tau_{n})$ lie in this $l$-wedge like space, $y_{1}>0,\ldots,y_{l}>0$.
The point $$\zeta_{x_{k}}:=(y_{1},\ldots,-y_{k},\ldots,y_{m},\tau_{1},...,\tau_{n})$$
is said to be symmetric for the point $\zeta$ with respect to the hyperplane $x_{k}=0.$
Similarly, the point $$\zeta_{x_{k}x_{s}}:=(y_{1},\ldots,-y_{k},\ldots,-y_{s},\ldots,y_{m},\tau_{1},...,\tau_{n})$$
is said to be symmetric for the point $\zeta_{x_{k}}$ with respect to the hyperplane $x_{s}=0$ and
so on. It is clear that the symmetry indices are invariant under permutations.
We will also need the notation $\Gamma((\zeta_{(j,l)})^{-1}\circ\xi),\;j\leq l,$ which means sum of the functions $\Gamma((\zeta_{(j,l)})^{-1}\circ\xi),\;j\leq l,$ over
all possible $(j,l)$ combination symmetry arguments: here in order to reduce the number of subindices we write $(\zeta_{(j,l)})^{-1}\circ\xi$ for
$\zeta^{-1}_{x_{k_{1}} \ldots x_{k_{j}}}\circ\xi$.
For example, if $l=3,\,j=2,$ then
$$\Gamma((\zeta_{(2,3)})^{-1}\circ\xi)=\Gamma((\zeta_{x_{1}x_{2}})^{-1}\circ\xi)+
\Gamma((\zeta_{x_{1}x_{3}})^{-1}\circ\xi)+\Gamma((\zeta_{x_{3}x_{2}})^{-1}\circ\xi),$$
and if $l=3,\,j=3,$ then
$$\Gamma((\zeta_{(3,3)})^{-1}\circ\xi)=\Gamma((\zeta_{x_{1}x_{2}x_{3}})^{-1}\circ\xi).$$

We have
\begin{prop}\label{greengen}
The function
\begin{equation}\label{greenfunctiongen}
G_{\mathbb{G}^{\ddagger}}(\xi,\zeta)=\Gamma(\zeta^{-1}\circ\xi)+\sum_{j=1}^{l}(-1)^{j}\Gamma((\zeta_{(j,l)})^{-1}\circ\xi)
\end{equation}
is the Green function for the Dirichlet sub-Laplacian in $\mathbb{G}^{\ddagger}$.
\end{prop}
\begin{proof}[Proof of Proposition \ref{greengen}]
Since by definition $\zeta_{(j,l)}
\not\in\mathbb{G}^{\ddagger},\,j=1,\ldots,l,$
it follows from \eqref{fundsoldelta} that
$$
\mathcal{L}\Gamma((\zeta_{(j,l)})^{-1}\circ\xi)
=-\delta_{\zeta_{(j,l)}}=0 \textrm{ in } \mathbb{G}^{\ddagger},
$$
for any $\xi\in\mathbb{G}^{\ddagger}$ and $j=1,\ldots,l$.
Thus, the function $\sum_{j=1}^{l}(-1)^{j}\Gamma((\zeta_
{(j,l)})^{-1}\circ\xi)$ satisfies the condition
\eqref{sub-harmonic}, i.e. it is harmonic
in $\mathbb{G}^{\ddagger}$.
Now it is left to check the boundary condition for the
domain $\mathbb{G}^{\ddagger}$, that is,
the function $G_{\mathbb{G}^{\ddagger}}$ should become
zero at $x_{1}=0$ and at infinity.
Recall that
\begin{equation}\label{distance}
d(\xi,\zeta):=\left(\Gamma(\zeta^{-1}\circ\xi)\right)^{\frac{1}{2-Q}}
\end{equation}
is an actual distance on $\mathbb{G}$ (see, e.g. \cite{C81}).
 Now it is easy to see that
the $d$-distance from any point of the hyperplane
$x_{k}=0$ to the points $\zeta$ and $\zeta_{x_{k}}$ is the same, that is, $G_{\mathbb{G}^{\ddagger}}$ satisfies the Dirichlet condition at the hyperplanes $x_{1}=0,\ldots,x_{l}=0$ and it is also clear (by the construction) that the function $G_{\mathbb{G}^{\ddagger}}$ is zero at the infinity. It proves that
$$G_{\mathbb{G}^{\ddagger}}(\xi,\zeta)=0,\quad \xi\in\partial\mathbb{G}^{\ddagger}.$$
\end{proof}

Now we consider a smooth open set $D\subset \mathbb{G}$ with boundary $\partial D$, and study the Dirichlet problem for the sub-Laplacian $\mathcal{L}$ in $D$.

For $0<\alpha<1$, Folland and Stein
(see \cite{FS} and see also \cite{Fol75}) defined the anisotropic H\"older spaces $\digamma_\alpha(D),$ $D\subset\mathbb{G},$
by
$$
\digamma_\alpha(D)=\{
f:D\to\mathbb C:\; \sup_{\stackrel{\xi,\zeta\in D}{\xi\not= \zeta}}
\frac{|f(\xi)-f(\zeta)|}{[d(\xi,\zeta)]^\alpha}<\infty
\},
$$
where $d$ is defined by the formula \eqref{distance} in our case.
For $k\in\mathbb N$ and $0<\alpha<1$, one defines
$\digamma_{k+\alpha}(D)$ as the space of all $f:D\to\mathbb C$ such that all $X_{j}$-derivatives of $f$ of order $k$ belong to $\digamma_\alpha(D)$.
A bounded function $f$ is called  $\alpha$-H\"{o}lder continuous
in $D\subset\mathbb{G}$ if $f\in \digamma_{\alpha}(D)$.

Let $f\in\digamma_{\alpha}(\mathbb{G}^{\ddagger}),\;0<\alpha<1,\; {\rm supp}\,f\subset\mathbb{G}^{\ddagger},$
and $\phi\in C^{\infty}(\partial\mathbb{G}^{\ddagger}),\; {\rm supp}\,\phi\subset\{x_{1}=0\}\bigcup\ldots \bigcup\{x_{l}=0\}.$
Consider the Dirichlet problem for the sub-Laplacian
\begin{equation}\label{LDgen-space}
\bigg\{\begin{matrix}
\mathcal{L}u=f\quad {\rm in}\; \mathbb{G}^{\ddagger}, \\
u=\phi\quad {\rm on}\; \partial \mathbb{G}^{\ddagger}.
\end{matrix}
\end{equation}

\begin{thm}\label{representationgen}
Let $f\in\digamma_{\alpha}(\mathbb{G}^{\ddagger}),\;0<\alpha<1,\; {\rm supp}\,f\subset\mathbb{G}^{\ddagger},$
and $\phi\in C^{\infty}(\partial\mathbb{G}^{\ddagger})$. Then the boundary value problem \eqref{LDgen-space} has a unique solution $u\in C^{2}(\mathbb{G}^{\ddagger})\cap C^{1}(\overline{\mathbb{G}^{\ddagger}})$ and
it can be represented by the formula
\begin{equation}\label{repr-space}
u(\xi)=\int_{\mathbb{G}^{\ddagger}}G_{\mathbb{G}^{\ddagger}}(\xi,\zeta)f(\zeta)d\nu(\zeta)
-\int_{\partial\mathbb{G}^{\ddagger}}\phi(\zeta)\langle \mathcal{\widetilde{\nabla}} G_{\mathbb{G}^{\ddagger}}(\xi,\zeta),\,d\nu(\zeta)\rangle,\quad \xi\in\mathbb{G}^{\ddagger},
\end{equation}
where $\mathcal{\widetilde{\nabla}}$ is defined by \eqref{nabla},
in particular,
$$\mathcal{\widetilde{\nabla}}G_{\mathbb{G}^{\ddagger}}=
\sum_{k=1}^{m}\left(X_{k}G_{\mathbb{G}^{\ddagger}}\right)X_{k}.$$
\end{thm}

\begin{proof}[Proof of Theorem \ref{representationgen}]
Let $u\in C^{2}(\mathbb{G}^{\ddagger})\cap C^{1}(\overline
{\mathbb{G}^{\ddagger}})$ and assume that $u$ tends to zero at infinity. The Green's second formula \eqref{g2} is in bounded domains, but it is still applicable for functions, with necessary decay rates at infinity, in unbounded domains. It can be shown by the standard argument using quasi-balls with radii $R\longrightarrow\infty.$ Thus, if we apply Green's second formula \eqref{g2} to the function $u$ with $v(\zeta)=G_{\mathbb{G}^{\ddagger}}(\xi,\zeta)$, we shall obtain
$$u(\xi)=\int_{\mathbb{G}^{\ddagger}}G_{\mathbb{G}^{\ddagger}}(\xi,\zeta)f(\zeta)d\nu(\zeta)
-\int_{\partial\mathbb{G}^{\ddagger}}\phi(\zeta)\langle \mathcal{\widetilde{\nabla}} G_{\mathbb{G}^{\ddagger}}(\xi,\zeta),\,d\nu(\zeta)\rangle.$$
Here we have used the properties of the Green function
$$
G_{\mathbb{G}^{\ddagger}}(\xi,\zeta)=0,\quad \zeta\in\partial \mathbb{G}^{\ddagger},
$$
and, by construction the function $G_{\mathbb{G}^{\ddagger}}$ is symmetric, that is,  $G_{\mathbb{G}^{\ddagger}}(\xi,\zeta)=G_{\mathbb{G}^{\ddagger}}(\zeta,\xi)$ in $\mathbb{G}^{\ddagger}$, so
$$
\mathcal{L}_{\zeta}G_{\mathbb{G}^{\ddagger}}(\xi,\zeta)=-\delta_{\xi},
$$
where $\delta_{\xi}$ is the Dirac distribution
at $\xi\in \mathbb{G}^{\ddagger}$.
Now we need to show that the function defined by \eqref{repr-space}
belongs to $C^{2}(\mathbb{G}^{\ddagger})\cap C^{1}(\overline{\mathbb{G}^{\ddagger}})$.
Since $f\in\digamma_{\alpha}(\mathbb{G}^{\ddagger}),\; {\rm supp}\,f\subset\mathbb{G}^{\ddagger},$ the volume potential (the first term of the right hand side in \eqref{repr-space}) belongs to $C^{2}(\overline{\mathbb{G}^{\ddagger}})$ by
Folland's theorem (see \cite[Theorem 6.1]{Fol75}, see also \cite{FS}).
H{\"o}rmander's hypoellipticity theorem (see \cite{H}) guarantees that every harmonic function is $C^{\infty}$, hence the Dirichlet double layer potential (the second term of the right hand side in \eqref{repr-space}) is in $C^{2}(\mathbb{G}^{\ddagger})$. On the other hand, since  $\phi\in C^{\infty}(\partial\mathbb{G}^{\ddagger}),\; {\rm supp}\,\phi\subset\{x_{1}=0,\ldots,x_{l}=0\}$ and the boundary hyperplanes $\{x_{1}=0\},\ldots,\{x_{l}=0\}$ have no characteristic points (see \cite[Section 8]{GV2000} for more discussions on the non-characteristic hyperplanes in $\mathbb{G}$) the Dirichlet double layer potential is continuous on the boundary by the Kohn-Nirenberg theorem (see \cite[Theorem 3.12]{CGH08}, which is a consequence of \cite[Theorem 4]{KN65}, see also \cite{De1}-\cite{De2}).
\end{proof}

\begin{rem}\label{greena}
 One may consider $\mathbb{G}^{\ddagger}_{a},\,a=(a_{1},\ldots,a_{l})\in\mathbb{R}^{l},$ $l$-wedge like space $\{\xi=(x_{1},\ldots,x_{m},t_{1},\ldots,t_{n})|\;x_{1}>a_{1},\ldots,x_{l}>a_{l}\}$, but in this $l$-wedge like space the Green function $G_{\mathbb{G}^{\ddagger}_{a}}$ has the same formula as the formula \eqref{greenfunctiongen} in which the symmetry points are chosen, in this case, with respect to the hyperplanes $\{x_{1}=a_{1}\},\ldots,\{x_{l}=a_{l}\}$. Of course, an analogue of Theorem \ref{representationgen} will be obtained by the same argument.
\end{rem}

Let us demonstrate some simple cases of Theorem \ref{representationgen} and Proposition \ref{greengen} with different (simpler) notations.
First, as above we construct a Green function for the Dirichlet sub-Laplacian in a half-space on $\mathbb{G}$. Let $\mathbb{G}^{+}$ be the half-space 
$$\mathbb{G}^{+}=\{\xi=(x_{1},...,x_{m},t_{1},...,t_{n})|\;x_{1}>0\}.$$
 Let the point $\zeta=(y,\tau)=(y_{1},y_{2},...,y_{m},\tau_{1},...,\tau_{n})$ lie in this half-space, $y_{1}>0$.
The point $$\zeta^{*}=(y^{*},\tau):=(-y_{1},y_{2},...,y_{m},\tau_{1},...,\tau_{n})$$
is said to be symmetric for the point $\zeta$ with respect to the hyperplane $x_{1}=0.$
We have the following direct consequence of Proposition \ref{greengen}.
\begin{cor}\label{green+}
The function
\begin{equation}\label{greenfunction+}
G_{\mathbb{G}^{+}}(\xi,\zeta)=\Gamma(\zeta^{-1}\circ\xi)-\Gamma((\zeta^{*})^{-1}\circ\xi)
\end{equation}
is the Green function for the Dirichlet sub-Laplacian in $\mathbb{G}^{+}$.
\end{cor}

Let $f\in\digamma_{\alpha}(\mathbb{G}^{+}),\; {\rm supp}\,f\subset\mathbb{G}^{+},$
and $\phi\in C^{\infty}(\partial\mathbb{G}^{+}),\; {\rm supp}\,\phi\subset\{x_{1}=0\}.$
Consider the Dirichlet problem for the sub-Laplacian
\begin{equation}\label{LDhalf-space}
\bigg\{\begin{matrix}
\mathcal{L}u=f\quad {\rm in}\; \mathbb{G}^{+}, \\
u=\phi\quad {\rm on}\; \partial \mathbb{G}^{+}.
\end{matrix}
\end{equation}

In this case Theorem \ref{representationgen} can be restated in the following form.
\begin{cor}\label{representation}
The boundary value problem \eqref{LDhalf-space} has a unique solution $u\in C^{2}(\mathbb{G}^{+})\cap C^{1}(\overline{\mathbb{G}^{+}})$ and
it can be represented by the formula
\begin{equation}\label{reprhalf-space}
u(\xi)=\int_{\mathbb{G}^{+}}G_{\mathbb{G}^{+}}(\xi,\zeta)f(\zeta)d\nu(\zeta)
-\int_{\partial\mathbb{G}^{+}}\phi(\zeta)\langle \mathcal{\widetilde{\nabla}} G_{\mathbb{G}^{+}}(\xi,\zeta),\,d\nu(\zeta)\rangle,\quad \xi\in\mathbb{G}^{+},
\end{equation}
where
$$G_{\mathbb{G}^{+}}(\xi,\zeta)=\Gamma(\zeta^{-1}\circ\xi)-\Gamma((\zeta^{*})^{-1}\circ\xi).$$
\end{cor}

Now we construct a Green function for the Dirichlet sub-Laplacian in a quadrant-space on $\mathbb{G}$.
Let $\mathbb{G}^{\oplus}$ be the quadrant-space 
$$
\mathbb{G}^{\oplus}=\{\xi=(x_{1},x_{2},...,x_{m},t_{1},...,t_{n})|\;x_{1}>0,\,x_{2}>0\}.
$$
Let the point $\zeta=(y,\tau)=(y_{1},y_{2},...,y_{m},\tau_{1},...,\tau_{n})$ lie in this quadrant-space, $y_{1}>0,\,y_{2}>0$. Denote by
$$\zeta^{*}=(y^{*},\tau):=(-y_{1},y_{2},...,y_{m},\tau_{1},...,\tau_{n})$$ and $$\overline{\zeta}=(\overline{y},\tau):=(y_{1},-y_{2},...,y_{m},\tau_{1},...,\tau_{n})$$
the symmetric points for $\zeta$ with respect to the hyperplanes $x_{1}=0$ and $x_{2}=0$, respectively. The point $$\overline{\zeta}^{*}=(\overline{y}^{*},\tau)=(-y_{1},-y_{2},...,y_{m},\tau_{1},...,\tau_{n})$$ is the symmetric point for $\zeta^{*}$ with respect to the hyperplane $x_{2}=0$ and the symmetric point for $\overline{\zeta}$ with respect to the hyperplane $x_{1}=0$.

We have the following another direct consequence of Proposition \ref{greengen}.
\begin{cor}\label{greeno+}
The function
\begin{equation}\label{greenfunctiono+}
G_{\mathbb{G}^{\oplus}}(\xi,\zeta)=\Gamma(\zeta^{-1}\circ\xi)+\Gamma((\overline{\zeta}^{*})^{-1}\circ\xi)-
\Gamma((\zeta^{*})^{-1}\circ\xi)-\Gamma((\overline{\zeta})^{-1}\circ\xi)
\end{equation}
is the Green function for the Dirichlet sub-Laplacian in $\mathbb{G}^{\oplus}$.
\end{cor}

\subsection{Green functions and representations of solutions in $l$-strip spaces.}
Let $\mathbb{G}^{\models}$ be the $l$-strip like space 
$$
\mathbb{G}^{\models}=\{\xi=(x_{1},\ldots,x_{m},t_{1},\ldots,t_{n})|\;a>x_{l}>0\},$$ for some
$1\leq l\leq m$. Let the point $\zeta=(y,\tau)=(y_{1},\ldots,y_{m},\tau_{1},\ldots,\tau_{n})$ lie in this $l$-strip space, $a>y_{l}>0$.
We will use the notations
$$\zeta_{+,j}:=(y_{1},\ldots,y_{l}-2aj,\ldots,y_{m},\tau_{1},...,\tau_{n}),$$
and
$$\zeta_{-,j}:=(y_{1},\ldots,-y_{l}+2aj,\ldots,y_{m},\tau_{1},...,\tau_{n}),$$
for all $j=0,1,2,\ldots.$
Following familiar pattern as above we obtain
\begin{prop}\label{Gstrip}
The function
\begin{equation}\label{Gfunctionstr}
G_{\mathbb{G}^{\models}}(\xi,\zeta)=\sum_{j=-\infty}^{\infty}\left(\Gamma(\zeta^{-1}_{+,j}\circ\xi)-
\Gamma(\zeta^{-1}_{-,j}\circ\xi)\right)
\end{equation}
is the Green function for the Dirichlet sub-Laplacian in $\mathbb{G}^{\models}$.
\end{prop}
\begin{proof}[Proof of Proposition \ref{Gstrip}]
It is evident that the first additive component in the $j=0$ term
of \eqref{Gfunctionstr}, i.e. the term $\Gamma(\zeta^{-1}_{+,0}\circ\xi)$ represents the fundamental solution and all the other terms are subharmonic functions in $\mathbb{G}^{\models}$.
Let us check that traces of \eqref{Gfunctionstr} vanish on hyperplanes $x_{l}=0$ and  $x_{l}=a$.
If $x_{l}=0$, then \eqref{Gfunctionstr} gives
\begin{multline}
G_{\mathbb{G}^{\models}}(\xi,\zeta)|_{x_{l}=0}=
\\c\sum_{j=-\infty}^{\infty}
\Big(\left(((x_{1}-y_{1})^{2}+\ldots+(-y_{l}+2aj)^{2}+\ldots+(x_{m}-y_{m})^{2})^{2}+16|t-\tau|^{2}\right)^{(2-Q)/4}
\\-\left(((x_{1}-y_{1})^{2}+\ldots+(y_{l}-2aj)^{2}+\ldots+(x_{m}-y_{m})^{2})^{2}+16|t-\tau|^{2}\right)^{(2-Q)/4}\Big)=0.
\end{multline}
If $x_{l}=a$, then \eqref{Gfunctionstr} gives
\begin{multline}
G_{\mathbb{G}^{\models}}(\xi,\zeta)|_{x_{l}=a}=
\\c\sum_{j=-\infty}^{\infty}
\Big(\left(((x_{1}-y_{1})^{2}+\ldots+(a-y_{l}+2aj)^{2}+\ldots+(x_{m}-y_{m})^{2})^{2}+16|t-\tau|^{2}\right)^{(2-Q)/4}
\\-\left(((x_{1}-y_{1})^{2}+\ldots+(a+y_{l}-2aj)^{2}+\ldots+(x_{m}-y_{m})^{2})^{2}+16|t-\tau|^{2}\right)^{(2-Q)/4}\Big)
=
\\c\sum_{j=0}^{\infty}
\left(((x_{1}-y_{1})^{2}+\ldots+(a-y_{l}+2aj)^{2}+\ldots+(x_{m}-y_{m})^{2})^{2}+16|t-\tau|^{2}\right)^{(2-Q)/4}
\\-c\sum_{j=1}^{\infty}\left(((x_{1}-y_{1})^{2}+\ldots+(a+y_{l}-2aj)^{2}+\ldots+(x_{m}-y_{m})^{2})^{2}+16|t-\tau|^{2}\right)^{(2-Q)/4}
\\+c\sum_{j=-1}^{-\infty}
\left(((x_{1}-y_{1})^{2}+\ldots+(a-y_{l}+2aj)^{2}+\ldots+(x_{m}-y_{m})^{2})^{2}+16|t-\tau|^{2}\right)^{(2-Q)/4}
\\-c\sum_{j=0}^{-\infty}\left(((x_{1}-y_{1})^{2}+\ldots+(a+y_{l}-2aj)^{2}+\ldots+(x_{m}-y_{m})^{2})^{2}+16|t-\tau|^{2}\right)^{(2-Q)/4}
=0.
\end{multline}
Here the first term ($j=0$ term) of the first sum is canceled with the first term ($j=1$ term) of the second sum and the second terms of the first sum is canceled with the second term of the second sum and so on, that is, the first two sums give zero. Similarly, the first term of the third sum is canceled with the first term of the last sum and the second term of the the third sum is canceled with the second term of the last sum and on on, that is, the last two sums also give zero. As a result, the trace vanishes at $x_{l}=a$.
\end{proof}

Let $f\in\digamma_{\alpha}(\mathbb{G}^{\models}),\;0<\alpha<1,\; {\rm supp}\,f\subset\mathbb{G}^{\models},$
and $\phi\in C^{\infty}(\partial\mathbb{G}^{\models}),\; {\rm supp}\,\phi\subset\{x_{l}=0\}\bigcup\{x_{l}=a\}.$
Consider the Dirichlet problem for the sub-Laplacian
\begin{equation}\label{LDgen-space}
\bigg\{\begin{matrix}
\mathcal{L}u=f\quad {\rm in}\; \mathbb{G}^{\models}, \\
u=\phi\quad {\rm on}\; \partial \mathbb{G}^{\models}.
\end{matrix}
\end{equation}

\begin{thm}\label{repstrip}
Let $f\in\digamma_{\alpha}(\mathbb{G}^{\models}),\;0<\alpha<1,\; {\rm supp}\,f\subset\mathbb{G}^{\models},$
and $\phi\in C^{\infty}(\partial\mathbb{G}^{\models})$. Then the boundary value problem \eqref{LDgen-space} has a unique solution $u\in C^{2}(\mathbb{G}^{\models})\cap C^{1}(\overline{\mathbb{G}^{\models}})$ and
it can be represented by the formula
\begin{equation}\label{repr-space}
u(\xi)=\int_{\mathbb{G}^{\models}}G_{\mathbb{G}^{\models}}(\xi,\zeta)f(\zeta)d\nu(\zeta)
-\int_{\partial\mathbb{G}^{\models}}\phi(\zeta)\langle \mathcal{\widetilde{\nabla}} G_{\mathbb{G}^{\models}}(\xi,\zeta),\,d\nu(\zeta)\rangle,\quad \xi\in\mathbb{G}^{\models},
\end{equation}
where $\mathcal{\widetilde{\nabla}}$ is defined by \eqref{nabla},
in particular,
$$\mathcal{\widetilde{\nabla}}G_{\mathbb{G}^{\models}}=
\sum_{k=1}^{m}\left(X_{k}G_{\mathbb{G}^{\models}}\right)X_{k}.$$
\end{thm}

\begin{proof}[Proof of Theorem \ref{repstrip}]
Similar to proof of Theorem \ref{representationgen}.
\end{proof}


\begin{thebibliography}{NZW01}

\bibitem[BLU07]{BLU07} A.~Bonfiglioli, E.~Lanconelli and F.~Uguzzoni.
\newblock  Stratified Lie groups and potential theory for their sub-Laplacians.
\newblock Springer- Verlag, Berlin, Heidelberg, 2007.

\bibitem[CGH08]{CGH08} L. Capogna, N. Garofalo, and D. Nhieu.
\newblock Mutual absolute continuity of harmonic
and surface measures for H\"ormander type operators.
\newblock
Perspectives in partial differential equations, harmonic analysis and applications,
49--100, Proc. Sympos. Pure Math., 79, Amer. Math. Soc., Providence, RI, 2008.

\bibitem[Cy81]{C81}
J. Cygan.
\newblock Subadditivity of homogeneous norms on certain nilpotent {L}ie groups.
\newblock {\em Proc. Amer. Math. Soc.}, 83:69--70, 1981.

\bibitem[DGN06]{DGN} D.~Danielli, N.~Garofalo and D.~Nhieu.
\newblock  {Non-doubling Ahlfors measures, perimeter measures, and the characterization of the trace spaces of Sobolev functions in Carnot-Carath{\'e}odory spaces.}
\newblock {\em Memoirs of the AMS}, Vol. 182, 2006.

\bibitem[De71]{De1}
M. Derridj.
\newblock Un probleme aux limites pour une classe d'operateurs du second ordre hypoelliptiques.
\newblock {\em Ann. Inst. Fourier (Grenoble)}, 21(4):99--148, 1971.

\bibitem[De72]{De2}
M. Derridj.
\newblock Sur un theoreme de traces.
\newblock {\em Ann. Inst. Fourier (Grenoble)}, 22(2):73--83, 1972.

\bibitem[DKM16]{DKM16}
S. Dubey, A. Kumar, and M.~M. Mishra.
\newblock The Neumann problem for the Kohn-Laplacian on the Heisenberg group $\mathbb{H}_{n}$.
\newblock {\em Potential Anal.}, published online, 2016.


\bibitem[FR16]{FR16}
V.~Fischer and M.~Ruzhansky.
\newblock {\em Quantization on nilpotent Lie groups}, volume 314 of {\em
  Progress in Mathematics}.
\newblock Birkh\"auser, 2016. (open access book)

\bibitem[Fol75]{Fol75}
G.~B. Folland.
\newblock Subelliptic estimates and function spaces on nilpotent {L}ie groups.
\newblock {\em Ark. Mat.}, 13(2):161--207, 1975.


\bibitem[FS74]{FS} G.~B.~Folland and E.~M.~Stein.
\newblock Estimates for the $\overline{\partial_{b}}$
complex and analysis on the Heisenberg group.
\newblock {\em Comm. Pure Appl. Math.}, 27:429--522, 1974.

\bibitem[GV00]{GV2000} N.~Garofalo and D.~Vassilev.
\newblock Regularity near the characteristic set in the non-linear Dirichlet problem and conformal geometry of sub-Laplacians on Carnot groups.
\newblock {\em Math. Ann.}, 318:453--516, 2000.

\bibitem[FS74]{FS} G.~B.~Folland and E.~M.~Stein.
\newblock Estimates for the $\overline{\partial_{b}}$
complex and analysis on the Heisenberg group.
\newblock {\em Comm. Pure Appl. Math.}, 27:429--522, 1974.

\bibitem[GL03]{GL03} C.~E.~Guti{\'e}rrez and E.~Lanconelli.
\newblock Maximum principle, nonhomogeneous Harnack inequality, and Liouville theorems for X-elliptic operators.
\newblock {\em Comm. Partial Differential Equations}, 28:1833--1862, 2003.



\bibitem[GN88]{GN88} C.~E.~Guti{\'e}rrez and G.~S.~Nelson.
\newblock Bounds for the fundamental solution of degenerate parabolic equations.
\newblock {\em Comm. Partial Differential Equations}, 13:635--649, 1988.

\bibitem[GW92]{GW92} C.~E.~Guti{\'e}rrez and R.~L.~Wheeden.
\newblock Bounds for the fundamental solution of degenerate parabolic equations.
\newblock {\em Comm. Partial Differential Equations}, 17:1287--1307, 1992.

\bibitem[Je81]{J}
D.~S.~Jerison.
\newblock The Dirichlet problem for the Kohn Laplacian on the Heisenberg group. I and II.
\newblock {\em J. Funct. Anal.}, 43:97--142, 224--257, 1981.

\bibitem[Ho67]{H}
H. H{\"o}rmander.
\newblock Hypoelliptic second-order differential equatio
ns.
\newblock {\em Acta Math.}, 43:147--171, 1967.

\bibitem[Kap80]{Kap80}
A.~Kaplan.
\newblock Fundamental solutions for a class of hypoelliptic PDE generated by composition
of quadratic forms.
\newblock {\em Trans. Amer. Math. Soc.}, 258:147--153, 1980.

\bibitem[KN65]{KN65}
J. J. Kohn and L. Nirenberg.
\newblock Non-coercitive boundary value problems.
\newblock {\em Comm. Pure and Appl. Math.}, 18(18):443--492, 1965.

\bibitem[LU97]{LU97} E.~Lanconelli and F.~Uguzzoni.
\newblock On the Poisson kernel for the Kohn Laplacian.
\newblock {\em Rend. Mat. Appl.}, 17(4):659--677, 1997.

\bibitem[RS16a] {Ruzhansky-Suragan:Kohn-Laplacian}
M.~Ruzhansky and D.~Suragan.
\newblock On {K}ac's principle of not feeling the boundary for the {K}ohn {L}aplacian on the {H}eisenberg group.
\newblock {\em Proc. Amer. Math. Soc.}, 144(2):709--721, 2016.

\bibitem[RS16b]{Ruzhansky-Suragan:squares}
M.~Ruzhansky and D.~Suragan.
\newblock Local {H}ardy and Rellich inequalities for sums of squares of vector fields.
\newblock {\em Adv. Diff. Equations}, to appear, 2016. arXiv:1605.06389

\bibitem[RS17a]{RS17a}
M.~Ruzhansky and D.~Suragan.
On horizontal Hardy, Rellich, Caffarelli-Kohn-Nirenberg and $p$-sub-Laplacian inequalities on stratified groups.
{\em J. Differential Equations}, 262:1799--1821, 2017.

\bibitem[RS17b]{Ruzhansky-Suragan:Layers}
M.~Ruzhansky and D.~Suragan.
\newblock Layer potentials, {K}ac's problem, and refined {H}ardy inequality on
  homogeneous {C}arnot groups.
\newblock {\em Adv. Math.}, 308:483--528, 2017.


\bibitem[WN16]{WN2016}
L.~Wu and P.~Niu.
\newblock 
Green functions for weighted subelliptic p-Laplace operator constructed by H{\"o}rmander’s vector fields.
\newblock {\em J. Math. Anal. Appl.}, 
444:1565--1590, 2016.


\end{thebibliography}
\end{document}